\documentclass[12pt,reqno]{amsart}

\usepackage[english]{babel}
\usepackage[utf8]{inputenc}
\usepackage{amsmath, amsthm, amssymb}
\usepackage{amsmath}
\usepackage{amsfonts}
\usepackage{amsthm}
\usepackage{mathrsfs}
\usepackage{amsrefs}
\usepackage{hyperref} 

\usepackage{xcolor}

\usepackage[titletoc]{appendix}

\theoremstyle{plain}
\newtheorem{theorem}{Theorem}[section]

\newtheorem{proposition}[theorem]{Proposition}

\theoremstyle{definition}

\newtheorem{remark}[theorem]{Remark}

\numberwithin{equation}{section}

\newcommand{\R}{\mathbb{R}}
\newcommand{\e}{\varepsilon}
\newcommand{\abs}[1]{\left\lvert#1\right\rvert}
\newcommand{\plaplacian}{\Delta_p}
\newcommand{\norm}[1]{\|{#1}\|}

\begin{document}
	
\title[Boundedness of stable solutions to nonlinear equations]{Boundedness of stable solutions to nonlinear equations involving the $p$-Laplacian}

\author{Pietro Miraglio}

\address{P. Miraglio \textsuperscript{1,2},
	\newline 
	\textsuperscript{1} Universit\`a degli Studi di Milano, Dipartimento di Matematica, via Cesare Saldini 50, 20133 Milano, Italy
	\newline
	\textsuperscript{2} Universitat Polit\`ecnica de Catalunya, Departament
de Matem\`{a}tiques, Diagonal 647,
08028 Barcelona, Spain}
\email{pietro.miraglio@upc.edu}



\thanks{P.M. is supported by the MINECO grant MTM2017-84214-C2-1-P and is part of	the Catalan research group 2017SGR1392.
}

\begin{abstract} We consider stable solutions to the equation $ -\Delta_p u =f(u) $ in a smooth bounded domain $ \Omega \subset \R^n $ for a $ C^1 $ nonlinearity $f$. Either in the radial case, or for some model nonlinearities $f$ in a general domain, stable solutions are known to be bounded in the optimal dimension range $n<p+4p/(p-1)$. In this article, under a new condition on $n$ and $p$, we establish an $ L^\infty $ a priori estimate for stable solutions which holds for every $ f\in C^1$. Our condition is optimal in the radial case for $n\geq3$, whereas it is more restrictive in the nonradial case. This work improves the known results in the topic and gives a unified proof for the radial and the nonradial cases.
	
The existence of an $L^\infty$ bound for stable solutions holding for all $C^1$ nonlinearities when $n<p+4p/(p-1)$ has been an open problem over the last twenty years. The forthcoming paper \cite{CMS} by Cabré, Sanchón, and the author will solve it when $p>2$.
\end{abstract}

\maketitle
\vspace{-0.5cm}
\section{Introduction}
For a smooth bounded domain $ \Omega\subset\R^n $, a $ C^1 $ nonlinearity $f$ and for every $ p\in(1,+\infty)$, we consider the elliptic equation involving the $p$-Laplacian
\begin{equation}\label{plap_eq}
-\plaplacian u=-\text{div}(\abs{\nabla u}^{p-2}\nabla u)= f(u) \qquad \text{in}\,\,\Omega,
\end{equation}
and the associated Dirichlet problem
\begin{equation}
	\label{plap_Dir}
	\left\{
	\begin{array}{rcll}
		-\plaplacian u&=& f(u) & \qquad \text{in}\,\,\Omega \\
		u&>&0 & \qquad\text{in}\,\,\Omega\\
		u&=&0 & \qquad\text{on}\,\,\partial\Omega.\\
	\end{array}\right.
\end{equation}
Solutions $u\in W^{1,p}(\Omega)$ to equation \eqref{plap_eq} correspond to critical points of the functional
\[
\mathcal{E}(u):=\int_\Omega \left(\frac1p\abs{\nabla u}^p-F(u)\right)\,dx,
\]
where $ F'(t)=f(t) $, and the boundary condition in \eqref{plap_Dir} is intended in the weak sense as $ u\in W^{1,p}_0(\Omega)$. Stable solutions to \eqref{plap_eq} are those for which the second variation of energy $ \mathcal{E} $ is nonnegative. 
More precisely, a solution $u\in C^{1}(\overline{\Omega})$ to \eqref{plap_eq} is said to be stable if
\begin{equation*}
\int_{\Omega\cap\left\{\abs{\nabla u}>0\right\}}\left\{\abs{\nabla u}^{p-2}\abs{\nabla\xi}^2+(p-2)\abs{\nabla u}^{p-4}\left(\nabla u\cdot\nabla\xi\right)^2\right\}\,dx-\int_\Omega f'(u)\xi^2\,dx\geq0
\end{equation*}
for every $ \xi\in \mathcal{T}_u $, defined in \cite{CasSa,FSV} as
\begin{equation*}
\mathcal{T}_u:=
\begin{aligned}
\begin{cases}
W_{\sigma,0}^{1,2}(\Omega)\qquad&\text{if}\,\,p\geq2;
\\
\{\xi\in W_0^{1,2}(\Omega):\,\,\norm{\nabla\xi}_{L_\sigma^2(\Omega)}<\infty	\} \qquad&\text{if}\,\,p\in(1,2).
\end{cases}
\end{aligned}
\end{equation*}
Here and throughout the paper, $\norm{\cdot}_{L_\sigma^2(\Omega)}$ is the weighted $L^2(\Omega)$ norm with weight~$\sigma=\abs{\nabla u}^{p-2}$, and $W_{\sigma,0}^{1,2}(\Omega)$ is defined as the completion of $C_c^1(\Omega)$ with respect to the norm
\begin{equation*}
\begin{split}
\norm{\xi}_{W_\sigma^{1,2}(\Omega)}:&=\norm{\xi}_{L^2(\Omega)}+\norm{\nabla\xi}_{L_\sigma^2(\Omega)}
\\
&=\left(\int_\Omega\xi^2\,dx\right)^\frac12+\left(\int_\Omega\abs{\nabla u}^{p-2}\abs{\nabla\xi}^2\,dx	\right)^\frac12.
\end{split}
\end{equation*}
See the beginning of section 4 in \cite{CasSa} for more details about the class $\mathcal{T}_u$ of test functions. In short, it is important to stress that, with this definition, $\mathcal{T}_u$ is a Hilbert space. The difference in defining the class is due to the fact that if $ p\geq2 $ then $ \sigma\in L^\infty(\Omega) $, while this is not true when $p\in(1,2)$.

We say that $u$ is a \textit{regular solution} to \eqref{plap_Dir} if it solves the equation in the distributional sense and $ f(u)\in L^\infty(\Omega)$.
Every regular solution is proved to be~$ C^{1,\beta}(\overline{\Omega}) $ --- see \cite{DiB,L,T} --- and
this is the best regularity that one can expect for solutions to nonlinear equations involving the $p$-Laplacian.

In this paper we focus on the boundedness of stable solutions to \eqref{plap_eq}, or to the associated Dirichlet problem \eqref{plap_Dir}, for general nonlinearities $ f\in C^1$. The importance of this problem for the classical Laplacian --- when $ p=2 $ --- has been stressed by Haïm Brezis since the mid-nineties --- see \cite{B,BV}. Very recently, it has been completely solved by Cabré, Figalli, Ros-Oton, and Serra \cite{CFRS}, proving that stable solutions are bounded whenever $n\leq9$. This result is indeed optimal, since explicit examples of unbounded stable solutions to \eqref{plap_eq} with $p=2$ are well-known when $n\geq10$.

The boundedness of stable solutions to \eqref{plap_Dir} is conjectured to hold under the assumption $ n<p+4p/(p-1) $. In fact, when $ n\geq p+4p/(p-1)$, $\Omega$ is a ball and $ f(u)=e^u $, García Azorero, Peral, and Puel showed in \cite{GPP} the existence of an unbounded stable solution to~\eqref{plap_Dir}. On the other hand, considering radial solutions to \eqref{plap_Dir} in a ball, Cabré, Capella, and Sanchón proved in \cite{CCaSa} the boundedness of stable solutions, provided that $ n<p+4p/(p-1) $. In the nonradial case and for general nonlinearities, the optimal result will be achieved in the forthcoming paper~\cite{CMS} by Cabré, Sanchón, and the author, assuming that $p>2$ and the domain is strictly convex. This is done extending to the $p$-Laplacian framework  some of the techniques used in~\cite{CFRS}. In the present work we do not use any idea or method developed in \cite{CFRS}.
The papers \cite{GPP, CCaSa, CMS} are part of an extensive literature on the topic, which is outlined in subsection \ref{subsec_available}.

The aim of the present paper is to provide $ L^\infty $ a priori bounds for stable solutions to~\eqref{plap_Dir} under a certain condition over $n$ and $p$. 
In the nonradial case, our condition over $n$ and $p$ is not optimal, but it improves the known results in the field.
In the radial case, our proof gives for $n\geq3$ the optimal result in \cite{CCaSa} in an unified way with the one in general domains.
Furthermore, our technique is based on a geometric Hardy inequality on the level sets of the stable solution. 
This approach --- that we explain below in detail --- has been introduced by Cabré in \cite{C1} for the classical version $ p=2 $ of the problem and it has never been used before in the context of the $ p $-Laplacian.

\subsection{Available results}\label{subsec_available}Let us describe first the large literature for the classical case $ p=2 $, and then list the most important results for problem \eqref{plap_Dir} with $ p\in(1,+\infty) $.

The first paper about this topic is by Crandall and Rabinowitz in 1975 \cite{CR}, in which they study problem \eqref{plap_Dir} with $p=2$ for smooth nonlinearities $f$ satisfying
\begin{equation}\label{extremal_hp}
f(0)>0, \qquad f'\geq0, \qquad \lim_{t\to+\infty}\frac{f(t)}{t}=+\infty.
\end{equation}
These assumptions are verified for instance by exponential and power-type nonlinearities, as discussed in \cite{CR}.

Assuming that $f$ satisfies \eqref{extremal_hp}, we can introduce extremal solutions, which are nontrivial examples of stable solutions to \eqref{plap_Dir}, sometimes unbounded. In order to define them in the classical case, let us consider a positive parameter $\lambda>0 $ and the Dirichlet problem
\begin{equation}\label{extr_problem}
\left\{
\begin{array}{rcll}
-\Delta u&=& \lambda f(u) & \qquad \text{in}\,\,\Omega \\
u&=&0 & \qquad\text{on}\,\,\partial\Omega.\\
\end{array}\right.
\end{equation}
It is known the existence of an extremal parameter $ \lambda^*\in(0,+\infty)$ such that if $ \lambda\in(0,\lambda^*) $, then problem (\ref{extr_problem}) admits a regular solution $ u_\lambda $ which is minimal, while if~$ \lambda>\lambda^* $ then it admits no regular solution. In addition, the family $ \{u_\lambda\} $ is increasing in $ \lambda $, every $ u_\lambda $ is stable, and one can define the limit
\begin{equation}\label{extremal_def}
u^*:=\lim_{\lambda\rightarrow\lambda^*}u_\lambda.
\end{equation}

The function $ u^* $ is a weak\footnote{In the sense introduced by Brezis et al.\cite{BetA}: $u\in L^1(\Omega)$ is a weak solution of \eqref{extr_problem} if $f(u)\text{dist}(x,\partial\Omega)\in L^1(\Omega)$ and 
	\[
	\int_\Omega\left(u\Delta \varphi+\lambda f(u)\varphi\right)\,dx=0
	\]
	for every $\varphi\in C_0^2(\overline{\Omega})$.} solution of \eqref{extr_problem} with $ \lambda=\lambda^* $ and it is stable. Assuming also that $f$ is convex, $ u^* $ is the unique weak solution to \eqref{extr_problem} for $ \lambda=\lambda^* $. It is called the \emph{extremal solution} of problem \eqref{extr_problem} and its boundedness depends on the dimension, the domain and the nonlinearity. In \cite{BV} the authors raised several open question about the extremal solution, especially about its regularity, which can be deduced from its boundedness using classical tools in the theory of elliptic PDEs --- see also the open problems raised by Brezis in \cite{B}.

When $ f (u) = e^u $, Crandall and Rabinowitz prove in \cite{CR} that $ u^*\in L^\infty(\Omega)$ if $ n \leq 9 $, while $ u^*(x) = \log \abs{x}^{-2} $ when $\Omega=B_1$ and $ n \geq 10 $ --- see \cite{JL}. Similar results hold for $ f (u) = (1 + u)^m $, and also for functions $ f $ such that
\begin{equation}\label{strong_condition}
\lim_{t\to+\infty}\frac{f(t)f''(t)}{f'(t)} \qquad\text{exists},
\end{equation}
as proved also in \cite{CR}.

We will describe now some $ L^\infty$ a priori estimates which have been proved for the smooth stable solutions $ u_\lambda $ to \eqref{extr_problem} with $ \lambda<\lambda^* $, under different assumptions on $f$. The estimates are uniform in $ \lambda $ and they led, by letting $ \lambda\nearrow\lambda^* $, to the boundedness of the extremal solution. Since the proofs work for every smooth stable solution to~\eqref{plap_Dir} with $ p=2 $ under the same assumptions on $f$, we describe here the results in the framework of stable solutions to \eqref{plap_Dir} with $ p=2 $.

Nedev obtained in \cite{N} an $ L^\infty $ bound for stable solutions in dimensions $ n =2, 3 $, under the hypothesis that $f$ is convex and satisfies \eqref{extremal_hp}. 
Some years later, Cabré and Capella \cite{CCa} solved the radial case for every Lipschitz nonlinearity, proving the boundedness of stable solutions when $ \Omega= B_1 $ and $ n \leq 9 $. 

In 2010 Cabré \cite{C} proved
that in dimensions $ n \leq 4 $ stable solutions are
bounded in every convex domain and for every $ C^1 $ nonlinearity.  
A few years later, Villegas~\cite{V} removed the convexity hypothesis about $ \Omega $ when $n=4$, by further assuming that $f$ is convex. Its proof uses both the results in \cite{C} and \cite{N},

The proof in \cite{C} is rather delicate and it is based on the Michael-Simon and Allard inequality on the level sets of a stable solution $u$. The same result has been proved very recently in \cite{C1} by the same author, using this time a Hardy inequality on the level sets of $u$. This new method is not only simpler, but it also gives a unified proof of the radial case --- in the optimal dimension range $ n\leq9 $ --- and of the nonradial case if $ n =3, 4 $, obtaining boundedness of stable solutions to \eqref{plap_Dir} with $p=2$ when $ \Omega$ is convex. 

In \cite{C,C1}, the $ L^\infty $ a priori bounds for stable solutions are obtained through an estimate in the interior of the domain combined later with some estimates near the boundary. The interior bounds hold for every regular domain $ \Omega $ and do not depend on the values of the stable solutions at the boundary. On the contrary, in order to have boundary estimates, the author needs to consider stable solutions to the Dirichlet problem \eqref{plap_Dir} with $p=2$ and also to assume the convexity of $ \Omega $. In the present paper, we follow the strategy of the second paper, \cite{C1}, extending it to the case of the $ p$-Laplacian.

As we mentioned above, very recently Cabré, Figalli, Ros-Oton, and Serra \cite{CFRS} settled the problem, proving that stable solutions are bounded in dimension $n\leq9$. The interior regularity applies to every nonnegative $f$, while the global result requires $f$ to be nondecreasing and convex. This was done by the authors using new and different ideas from the ones in~\cite{C,C1}. For more details about the classical problem for the Laplacian we refer to the recent survey~\cite{C2} and to the book~\cite{D}.

Before outlining in detail our results, we comment on what is known about the boundedness of stable solutions for the $ p $-Laplacian. Let us start by describing the extremal solutions for the problem
\begin{equation}\label{p_extr_problem}
\left\{
\begin{array}{rcll}
-\plaplacian u&=& \lambda f(u) & \qquad \text{in}\,\,\Omega \\
u&=&0 & \qquad\text{on}\,\,\partial\Omega,\\
\end{array}\right.
\end{equation}
where $ \lambda $ is a positive parameter and $f$ a $C^1$ nonlinearity. Under the assumptions
\begin{equation}\label{p_extremal_hp}
f(0)>0, \qquad f'\geq0, \qquad \lim_{t\to+\infty}\frac{f(t)}{t^{p-1}}=+\infty,
\end{equation}
there exists an extremal parameter $ \lambda^*\in(0,+\infty)$ such that if $ \lambda\in(0,\lambda^*) $, then problem \eqref{p_extr_problem} admits a minimal regular solution $ u_\lambda $, while if $ \lambda>\lambda^* $ then it admits no regular solution.  Furthermore, the family $ \{u_\lambda\} $ is increasing in $ \lambda $, every $ u_\lambda $ is stable and we can define $u^*$ as in \eqref{extremal_def} --- see \cite{CSa} for these results about the extremal problem for the $p$-Laplacian.

For $p\neq2$ and $f\in C^1$ satisfying \eqref{p_extremal_hp}, it is not known in general whether $u^*$ is a distributional solution of \eqref{p_extr_problem} with $\lambda=\lambda^*$. 
However, when $f$ is the exponential nonlinearity or it satisfies some strong assumptions --- see \cite{GP,GPP,CSa,San1} --- it has been proved that $u^*$ is a distributional solution to \eqref{p_extr_problem} with $\lambda=\lambda^*$. In this cases, it is called the \textit{extremal solution} of problem \eqref{p_extr_problem}.

As in the classical case, also for $ p>1 $ the integrability and regularity properties of $u^*$ are obtained as a consequence of uniform estimates for the stable branch $ \{u_\lambda\} $.

García Azorero, Peral, and Puel treated the exponential nonlinearity $ f(u)=e^u $ for~$ p>1 $ in \cite{GP,GPP}. They established the boundedness of stable solutions when 
\begin{equation}\label{optimal_dimension}
n<p+\frac{4p}{p-1},
\end{equation} 
and showed that this condition is optimal. Indeed, they provided an example of unbounded stable solution to \eqref{plap_Dir} with $ f(u)=e^u $, $ \Omega=B_1 $ and $ n\geq p+4p/(p-1) $.

Some years later, Sanchón proved in \cite{San} that stable solutions are bounded in the optimal dimension range \eqref{optimal_dimension}, under the hypothesis that $f\in C^2$ is an increasing function, it satisfies \eqref{p_extremal_hp} and also the strong assumption \eqref{strong_condition} on the behavior of $f$ at infinity. The same result is obtained by Cabré and Sanchón in \cite{CSa}, assuming that the nonlinearity satisfies~\eqref{p_extremal_hp} and the power growth hypothesis $f(t)\leq c(1+t)^m$, where $m$ is smaller than a ``Joseph-Lundgren type" exponent which is optimal for the regularity of stable solutions.

As we mentioned above, the radial case of problem \eqref{plap_Dir} was settled by Cabré, Capella, and Sanchón in \cite{CCaSa} for every locally Lipschitz nonlinearity. Indeed, under this assumption they proved that radial stable solutions are bounded in the optimal range $n <p+4p/(p-1)$.

Back to the nonradial case, the following works deal with general nonlinearities satisfying essentially \eqref{p_extremal_hp}. They are also the most recent results in the topic.

Sanchón in \cite{San,San1} considers nonlinearities $f$ that satisfy \eqref{p_extremal_hp} and
\begin{equation}\label{p_convexity}
\text{there exists}\,\,\, T\geq0\,\,\,\text{such that}\,\,\,(f(t)-f(0))^\frac{1}{p-1} \,\,\, \text{is convex for all}\,\,\,t\geq T.
\end{equation}
Observe that when $p=2$ this last condition becomes the standard convexity assumption on $f$ made in \cite{BV} and appearing also in the recent paper \cite{CFRS}.

In \cite{San,San1} it is proved the boundedness of stable solutions whenever
\begin{equation}\label{San_condition}
\begin{cases}
n<p+\frac{p}{p-1}\qquad &\text{and}\qquad p\geq2;
\\
n\leq p+\frac{2p}{p-1}(1+\sqrt{2-p})\qquad &\text{and}\qquad p\in(1,2).
\end{cases}
\end{equation}
Both results are obtained following the approach of Nedev in \cite{N} for $p=2$. Later, Castorina and Sanch\'on \cite{CasSa} extended Cabré's method in \cite{C} for $p=2$ to the case of the $p$-Laplacian, proving that stable solutions are bounded in the range
\begin{equation}\label{Cas_San_condition}
n\leq p+2
\end{equation}
under the assumption that $f$ is $C^1$, and satisfies~\eqref{p_extremal_hp} and \eqref{p_convexity}.

In the forthcoming paper \cite{CMS} by Cabré, Sanchón, and the author, the interior results in \cite{CFRS} for $p=2$ will be extended to the case of the $p$-Laplacian. In particular, we will prove that stable solutions are bounded in the optimal dimension range $n<p+4p/(p-1)$ whenever $p>2$ and $\Omega$ is strictly convex.
\subsection{New results and strategy of the proof}
Theorem \ref{thm_main} below is the main result of the present article. It establishes, under a new condition on~$n$ and~$p$, an~$L^\infty$ a priori estimate for stable solutions for every $C^1$ nonlinearity. This condition improves the one in \cite{CasSa}, \eqref{Cas_San_condition}, when $n\geq4$ and $p>2$, even though it is not optimal.

Our result consists of an interior estimate for stable solutions which does not depend on the boundary values of the function and holds for every~$C^1$ nonlinearity and every bounded domain --- see \eqref{interior_bound} below. Up to our knowledge, ours is the first result of this kind for stable solutions to $ \eqref{plap_eq} $ in the setting of the $p$-Laplacian.

This interior estimate leads to a global $L^\infty$ estimate under the further assumption that the domain is strictly convex and that $u$ is a stable solution of the Dirichlet problem \eqref{plap_Dir}, and not only of equation \eqref{plap_eq} --- see \eqref{global_bound} below.

\begin{theorem}\label{thm_main}
	Let $f$ be any $ C^1 $ nonlinearity, $ \Omega\subset\R^n$ a bounded smooth domain, $ p\in(1,+\infty) $, and $u$ a regular stable solution to \eqref{plap_eq}. Assume that
	\begin{equation}\label{condition}
	\begin{aligned}
	&n\geq4 \quad\text{and}\quad n<\frac12\left(\sqrt{(p-1)(p+7)}+p+5\right);
	\\
	\text{or}\quad &n=3\quad\text{and}\quad p<3.
	\end{aligned} 	
	\end{equation}
	
	\begin{itemize}
		\item[(i)] Then, for every $ \delta>0 $, we have that
		\begin{equation}\label{interior_bound}
		\norm{u}_{L^\infty(K_\delta)}\leq C\left(\norm{u}_{L^1(\Omega)}+\norm{\nabla u}_{L^p(\Omega\setminus K_\delta)}\right),
		\end{equation}
		where
		\[
		K_\delta:=\left\{x\in\Omega:\, \textnormal{dist}(x,\partial\Omega)\geq\delta	\right\},
		\]
		and $C$ is a constant depending only on $ \Omega $, $ \delta $, and $p$.
		
		\item[(ii)] If in addition $ \Omega $ is strictly convex, $ u $ is a positive solution of the Dirichlet problem \eqref{plap_Dir}, and $f$ is positive, then
		\begin{equation}\label{global_bound}
		\norm{u}_{L^\infty(\Omega)}\leq C,
		\end{equation}
		where $ C $ is a constant depending only on $ \Omega $, $p$, $f$, and $\norm{u}_{L^1(\Omega)}$.
		
		\item[(iii)] If $ \Omega $ is a ball and $f$ is strictly positive in $(0,+\infty)$, then both \eqref{interior_bound} and \eqref{global_bound} hold if
		\begin{equation}\label{radial_condition}
		\begin{aligned}
			&n\geq3\quad\text{and}\quad n<p+\frac{4p}{p-1};
			\\
			&n=2\quad\text{and}\quad p\in(1,3).
		\end{aligned} 	
		\end{equation}
	\end{itemize}
\end{theorem}

\begin{remark}\label{rmk_condition}
	We point out that condition~\eqref{condition} forces $p>2$ for $n\geq5$, and  $p>4/3$ for $n=4$.
	Furthermore, our condition \eqref{condition} improves \eqref{Cas_San_condition} for~$n\geq4$, since $p+2<(\sqrt{(p-1)(p+7)}+p+5)/2$.
\end{remark}

The interior estimate~\eqref{interior_bound} does not require any assumption on the values of $u$ at the boundary of $\Omega$, nor the strict convexity of the domain.
On the other hand, passing from~\eqref{interior_bound} to the global bound~\eqref{global_bound} requires some boundary estimates, which are available if we assume that the domain is strictly convex, $u$ is a positive stable solution to the Dirichlet problem \eqref{plap_Dir}, and $f$ is positive. We will introduce the boundary estimates in Section~\ref{sec_proof}, before the proof of Theorem \ref{thm_main}.

Our main result Theorem \ref{thm_main} is obtained as a consequence of the following proposition. It is an estimate of the weighted $ L^p $ norm of the gradient of $u$ in $ \Omega $, being controlled by the $L^p$ norm of the gradient of $u$ in a small neighborhood of the boundary of the domain.

\begin{proposition}\label{prop_interior_estimate}
	Let $f$ be any $ C^1 $ nonlinearity, $ \Omega\subset\R^n$ a bounded smooth domain, $ p\in(1,+\infty) $, and $u$ a regular stable solution to \eqref{plap_eq}. Let $ \alpha\in[0,n-1)$ satisfy
	\begin{equation}\label{alpha_condition}
	\begin{aligned}
	&4(n-1-\alpha)^2>(\alpha-2)^2(n-1)(p-1) 
	\qquad&\text{if}\,\,n> p;
	\\
	& 4(n-1-\alpha)^2>(\alpha-2)^2(p-1)^2 
	\qquad&\text{if}\,\,n\leq p.
	\end{aligned} 	
	\end{equation}
	
	Then, for all $ \delta>0 $ and $ y\in K_\delta:=\{x\in\Omega :\,\textnormal{dist}(x,\partial\Omega)\geq\delta \} $, it holds that
	\begin{equation}\label{interior_ineq}
	\int_{\Omega}\abs{\nabla u(x)}^p\abs{x-y}^{-\alpha}\,dx\leq C\norm{\nabla u}^p_{L^p(\Omega\setminus K_\delta)},
	\end{equation}
	where $ C $ is a constant depending only on $ \Omega $, $ \delta $, $p$, and $ \alpha $.
	
	If $ \Omega $ is a ball and $f$ is strictly positive in $(0,+\infty)$, then \eqref{interior_ineq} holds with $y=0$ if, instead of \eqref{alpha_condition}, we assume that $ \alpha\in[0,n-1)$ satisfies
	\begin{equation}\label{radial_alpha_condition}
	\begin{aligned}
	&4(n-1)>(\alpha-2)^2(p-1) 
	\qquad&\text{if}\,\,n> p;
	\\
	&4(n-1)^2>(\alpha-2)^2(p-1)^2
	\qquad&\text{if}\,\,n\leq p.
	\end{aligned} 	
	\end{equation}
\end{proposition}

As we mentioned above, in order to prove Proposition \ref{prop_interior_estimate} we follow the strategy used in \cite{C1} for the problem with the Laplacian. The main ingredients are a geometric inequality for stable solutions to \eqref{plap_eq} and a Hardy inequality on the level sets of the function $u$. The first tool is originally due to Sternberg and Zumbrun \cite{SZ1,SZ2} for the case of the Laplacian. We will use a generalization of this inequality to the $ p $-Laplacian case, due to Farina, Sciunzi, and Valdinoci \cite{FSV,FSV1} and stated in Theorem~\ref{thm_SZ_stability} below.

The Hardy inequality that we use is originally due to Cabré \cite{C1}, but it can also be deduced from more general Hardy inequalities studied in \cite{CM} by Cabré and the author. In order to state these two results, we need to introduce some notation.

If $u$ is a $C^1$ solution to \eqref{plap_eq} and we consider the set of regular points of $u$, defined by $\{ x\in\Omega: \abs{\nabla u(x)}>0 \}$, then $u$ is $ C^2 $ in this set --- see Corollary 2.2 of \cite{DS} --- since the equation is uniformly elliptic in a neighborhood of every regular point.

Therefore, for any $x\in\Omega\cap\{\abs{\nabla u}>0\}$ we can define the level set of $u$ passing through $x$ as
\[\mathcal{L}_{u,x} :=\{y\in\Omega: u(y)=u(x)	\},\]
which is a $ C^2 $ embedded hypersurface of~$\R^n$.
In~$\{ x\in\Omega: \abs{\nabla u(x)}>0 \}$ we can define
\[
\nu:=\frac{\nabla u}{\abs{\nabla u}},
\]
which is the normal vector to the level sets of $ u $. Now, we can also introduce the notion of tangential gradient along the level sets. We define it for every function $ \varphi\in C^1(\Omega) $ as the projection of $ \nabla\varphi $ on the tangent space to the level sets passing through $ x $, i.e.
\[
\nabla_T\varphi:=\nabla\varphi-\left \langle \nabla \varphi, \nu \right \rangle \nu.
\]
For any $x\in\Omega\cap\{\abs{\nabla u}>0\}$ we denote with $ \kappa_i $ the $n-1$ principal curvatures of $ \mathcal{L}_{u,x} $ and
we recall that the mean curvature of the level sets is defined as
\[
H:=\sum_{i=1}^{n-1}\kappa_i.
\]
In the statement of the geometric property of stable solutions, the square of the second fundamental form of the level sets appears. It is defined as
\[
\abs{A}^2:=\sum_{i=1}^{n-1}\kappa^2_i.
\]
Now, we can state the geometric inequality for stable solutions to $ -\Delta_p u=f(u) $.
\begin{theorem}[Farina, Sciunzi, Valdinoci \cite{FSV, FSV1}]\label{thm_SZ_stability}
	Let $ p\in(1,+\infty)$, $ \Omega $ be a bounded smooth domain of $ \R^n $, $f$ a $ C^1 $ nonlinearity and $u$ a regular stable solution to~\eqref{plap_eq}. Then, for every $\eta\in C^1_c(\Omega)$ it holds that 
	\begin{equation}\label{SZ_formula}
	\begin{split}
	&\hspace{-2cm}\int_{\Omega\cap\{|\nabla u|>0\}}\Big((p-1)\abs{\nabla u}^{p-2}\abs{\nabla_T\abs{\nabla u}}^2+\abs{\nabla u}^p \abs{A}^2\Big)\eta^2\,dx 
	\\
	&\hspace{5cm}\leq (p-1)\int_{\Omega}\abs{\nabla u}^p\abs{\nabla \eta}^2\,dx.	
	\end{split}
	\end{equation}
\end{theorem}

As we mentioned above, this result is originally due to Sternberg and Zumbrun \cite{SZ1,SZ2} for stable solutions to $ -\Delta u=f(u) $ in a bounded smooth domain $ \Omega $ with $f\in C^1$. In this case, for every $\eta\in C^1_c(\Omega)$, the inequality reads
\begin{equation}\label{SZp=2}
\int_{\Omega\cap\{|\nabla u|>0\}}\Big(\abs{\nabla_T\abs{\nabla u}}^2+\abs{\nabla u}^2 \abs{A}^2\Big)\eta^2\,dx
\leq \int_{\Omega}\abs{\nabla u}^2\abs{\nabla \eta}^2\,dx.
\end{equation}

The idea of obtaining $ L^\infty $ bounds for stable solutions to $ -\Delta u=f(u) $ using \eqref{SZp=2} was used for the first time in \cite{C}. The key point in \cite{C} is the combination of \eqref{SZp=2} with the Michael-Simon and Allard inequality, applied on every level set of $u$.

A similar but simpler strategy is used in \cite{C1}, still for the classical problem with the Laplacian. It is based on a new geometric Hardy inequality on the foliation of hypersurfaces given by the level sets of $u$, a much simpler tool than the Michael-Simon and Allard inequality. 
In the present paper we extend this idea to the case of the $p$-Laplacian. 
We need both Theorem \ref{thm_SZ_stability} and the new Hardy inequality provided in~\cite{C1} to prove Proposition \ref{prop_interior_estimate}, which is the key estimate to prove Theorem \ref{thm_main}. 

We need to introduce some further notation in order to state the Hardy inequality on hypersurfaces of $\R^n$. For every $ y\in\R^n $, we define
\[
r_y=r_y(x)=\abs{x-y},
\]
and for every function $ \psi(x)\in C^1 $ we write its radial derivative as 
\[
\psi_{r_y}(x)=\frac{x-y}{\abs{x-y}}\cdot\nabla\psi(x).
\]
The geometric Hardy inequality is stated in the following theorem. Recall that, in the statement, the mean curvature $ H $ and the tangential gradient $ \nabla_T $ are referred to the level sets $ \mathcal{L}_{u,x} $ of $u$, which are $ C^2 $ embedded hypersurfaces of $ \R^n $ for every $x\in\Omega\cap\{\abs{\nabla u}>0\}$.
\begin{theorem}[Cabré \cite{C1}]\label{thm_hardy}
	Let $ \Omega\subset\R^n $ be a bounded smooth domain, $u\in C^1_c(\R^n)\cap C^2_c(\R^n\cap\{\abs{\nabla u}>0\})$, $ \alpha\in[0,n-1) $ and $ y\in\R^n $. Then, for every $ \varphi\in C^1_c(\R^n) $
	\begin{equation}\label{hardy_ineq}
	\begin{split}
	&\hspace{-0.1cm}(n-1-\alpha)\int_{\Omega}\abs{\nabla u}\varphi^2 r^{-\alpha}\,dx+\alpha\int_{\Omega}\abs{\nabla u}^{-1}u_{r_y}^2\varphi^2 r^{-\alpha}\,dx
	\\
	&\hspace{0.6cm}\leq \left(\int_{\Omega}\abs{\nabla u}\varphi^2 r^{-\alpha}\,dx\right)^\frac12\left(\int_{\Omega\cap\{|\nabla u|>0\}}\abs{\nabla u}\left(4\abs{\nabla_T\varphi}^2+\varphi^2\abs{H}^2\right)r^{-\alpha+2}\,dx\right)^\frac12.
	\end{split}
	\end{equation}
	In particular, if $ u $ is radial, then
	\begin{equation*}
	(n-1)^2\int_{\Omega}\abs{u_{r_y}}\varphi^2 r^{-\alpha}\,dx
	\leq \int_{\Omega\cap\{|u_{r_y}|>0\}}\abs{u_{r_y}}\left(4\abs{\nabla_T\varphi}^2+\varphi^2\abs{H}^2\right)r^{-\alpha+2}\,dx.
	\end{equation*}
\end{theorem}

\section{Proof of the $ L^\infty $ bounds}\label{sec_proof}

We prove in this section our main results, namely Proposition \ref{prop_interior_estimate} and Theorem~\ref{thm_main}, using the geometric inequality for stable solutions and the Hardy inequality on level sets.

\begin{proof}[Proof of Proposition \ref{prop_interior_estimate}]
		We apply the geometric Hardy inequality of Theorem \ref{thm_hardy} to the function
		\[\varphi=\abs{\nabla u}^\frac{p-1}{2}\zeta, \]
		where $ \zeta$ is a positive smooth function that satisfies
		\begin{equation}\label{zeta}
		\zeta_{|\partial\Omega}=0\qquad\text{and}\qquad\zeta\equiv1\,\,\,\text{in}\,\,\, K_{\delta/2}. 
		\end{equation}
		To be completely rigorous, in the proof we should use
		\[
		\varphi_\varepsilon:=\left(\abs{\nabla u}^2+\varepsilon^2\right)^\frac{p-1}{4}\zeta
		\]
		instead of $ \varphi $, and then let $ \varepsilon\to0 $. We omit the details of this simple argument.
		
		To simplify notation, we define
		\[
		I:=\int_\Omega\abs{\nabla u}^p r_y^{-\alpha}\zeta^2\,dx;
		\]
		\[
		I_r:=\int_\Omega\abs{\nabla u}^{p-2}u_{r_y}^2 r_y^{-\alpha}\zeta^2\,dx.
		\]
		Plugging $ \varphi $ into \eqref{hardy_ineq}, we obtain
		\begin{equation}\label{66}
		\begin{split}
		&\left((n-1-\alpha)I+\alpha I_r\right)^2 
		\\
		&\hspace{2cm}\leq I\int_{\Omega\cap\{|\nabla u|>0\}}\abs{\nabla u}r_y^{-\alpha+2}\left(4\abs{\nabla_T\varphi}^2+\abs{\nabla u}^{p-1}\zeta^2\abs{H}^2\right)\,dx,
		\end{split}
		\end{equation}
		with $\alpha\in[0,n-1)$ to be chosen. The tangential gradient of $ \varphi $ can be computed as
		\[
		\nabla_T\varphi=\frac{p-1}{2}\zeta\abs{\nabla u}^\frac{p-3}{2}\nabla_T\abs{\nabla u}+\abs{\nabla u}^\frac{p-1}{2}\nabla_T\zeta,
		\]
		and the Cauchy-Schwarz inequality gives
		\[
		4\abs{\nabla_T\varphi}^2
		\leq
		(1+\varepsilon)(p-1)^2\zeta^2\abs{\nabla u}^{p-3}\abs{\nabla_T\abs{\nabla u}}^2
		+\frac{C}{\varepsilon}\abs{\nabla u}^{p-1}\abs{\nabla_T\zeta}^2,
		\]
		where $C$ is a positive universal constant, and $\varepsilon>0$ will be chosen later. Therefore, we get
		\begin{equation}\label{643}
		\begin{split}
		&\left((n-1-\alpha)I+\alpha I_r\right)^2 \leq (1+\varepsilon)I\int_{\Omega\cap\{|\nabla u|>0\}}\Big((p-1)^2\abs{\nabla u}^{p-2}\abs{\nabla_T\abs{\nabla u}}^2
		\\
		&\hspace{2cm}+\abs{\nabla u}^pH^2\Big)r_y^{-\alpha+2}\zeta^2\,dx
		+\frac{C}{\varepsilon}I\int_{\Omega\cap\{|\nabla u|>0\}}\abs{\nabla u}^p r_y^{-\alpha+2}\abs{\nabla_T\zeta}^2\,dx.
		\end{split}
		\end{equation}
		Concerning the last integral in \eqref{643}, we can control it in terms of the $L^p$-norm of the gradient of $u$ in a neighborhood of the boundary of $\Omega$, since $\abs{\nabla_T\zeta}$ has support in $(\Omega\setminus K_{\delta/2})\subset(\Omega\setminus K_{\delta})$. We also use that, since $y\in K_\delta$, we have
		\begin{equation}\label{r_bound}
		\delta/2<r_y(x)<\textnormal{diam}(\Omega)	\qquad\text{for every}\,\,x\in\Omega\setminus K_{\delta/2}.
		\end{equation}
		Therefore, we deduce the bound
		\begin{equation}\label{644}
		\int_{\Omega\cap\{|\nabla u|>0\}}\abs{\nabla u}^p r_y^{-\alpha+2}\abs{\nabla_T\zeta}^2\,dx
		\leq
		C\int_{\Omega\setminus K_\delta}\abs{\nabla u}^p\,dx,
		\end{equation}
		for some positive constant $C$ depending only on $\Omega$, $\delta$, and $\alpha$. Observe that we need both the upper and the lower bound on $ r_y(x) $ since a priori $ \alpha $ in \eqref{644} can be greater or smaller than 2.
		
		In the next step, we use that $ H^2\leq(n-1)\abs{A}^2$ and apply the geometric stability inequality \eqref{SZ_formula} in Theorem~\ref{thm_SZ_stability}. Observe that, to apply it, we need to have $(p-1)$ instead of $(p-1)^2$ in the first term in the right-hand side of \eqref{643}, and no constants in front of the term containing $\abs{A}^2$.
		This will force us to make a bound which differs whether $n$ is above or below $p$. For this reason, we distinguish the two cases.
		
		When $n> p$, we have $p-1< n-1$ and from \eqref{643} we deduce that
		\begin{equation}\label{99}
		\begin{split}
		&\hspace{-1.5cm}\left((n-1-\alpha)I+\alpha I_r\right)^2	 
		\\
		&\leq (1+\varepsilon)(n-1)I\int_{\Omega\cap\{\abs{\nabla u}>0\}}\Big((p-1)\abs{\nabla u}^{p-2}\abs{\nabla_T\abs{\nabla u}}^2
		\\
		&\hspace{1.5cm}+\abs{\nabla u}^p\abs{A}^2\Big)r_y^{-\alpha+2}\zeta^2\,dx
		+\frac{C}{\varepsilon}I\int_{\Omega\setminus K_\delta}\abs{\nabla u}^p\,dx.
		\end{split}
		\end{equation}
		
		Now, we can control the right-hand side of \eqref{99} using the geometric stability inequality~\eqref{SZ_formula} with test function
		\begin{equation}\label{test_FSV}
		\eta=r_y^\frac{2-\alpha}{2}\zeta.
		\end{equation}
		The following computations must be done with a regularization of $\eta$ in a small neighborhood of $y$, that we call $\eta_\varepsilon$. Since all terms in the rest of the proof are given by integrable functions, by dominated convergence we can let $\varepsilon\to0$ in all the integrals. For this reason, we directly write the computations with $ \eta $ instead of $ \eta_\varepsilon $.

		Plugging $\eta $ in \eqref{SZ_formula} and combining it with \eqref{99}, we obtain
		\begin{equation*}
		\begin{split}
		&\left((n-1-\alpha)I+\alpha I_r\right)^2 
		\\ 
		&\hspace{1.5cm}\leq (1+\varepsilon)(n-1)(p-1)I\int_{\Omega}\abs{\nabla u}^p\abs{\nabla\left(r_y^\frac{2-\alpha}{2}\zeta\right)}^2dx+\frac{C}{\varepsilon}I\int_{\Omega\setminus K_\delta}\abs{\nabla u}^p\,dx.
		\end{split}
		\end{equation*}
		Using again the Cauchy-Schwarz inequality, there exists a positive universal constant $C$ such that
		\[
		\abs{\nabla\left(r_y^\frac{2-\alpha}{2}\zeta\right)}^2
		\leq
		(1+\varepsilon)\frac{(\alpha-2)^2}{4} r_y^{-\alpha}\zeta^2
		+\frac{C}{\varepsilon}r_y^{2-\alpha}\abs{\nabla\zeta}^2,		
		\]
		again for the same $\e>0$ that we will choose later. Since we have chosen $ \zeta $ satisfying \eqref{zeta}, if $ n>p $ we get
		\begin{equation}\label{56}
		\begin{split}
		\left(n-1-\alpha\right)^2I^2&\leq\left((n-1-\alpha)I+\alpha I_r\right)^2
		\\ 
		&\leq(1+\varepsilon)^2(n-1)(p-1)\frac{(\alpha-2)^2}{4}I^2+\frac{C}{\varepsilon}I\norm{\nabla u}^p_{L^p(\Omega\setminus K_\delta)}.
		\end{split}
		\end{equation}
		
	If instead $ n\leq p $, the same procedure works --- including the same choice of test function \eqref{test_FSV} in the stability inequality \eqref{SZ_formula} --- but we have a difference in the constants. Indeed, in \eqref{643} we use that $ H^2\leq(n-1)\abs{A}^2\leq(p-1)\abs{A}^2$. 
		In this way, we can take $(p-1)$ out of the integral and obtain the right constants to apply the geometric stability inequality \eqref{SZ_formula}. As a consequence, instead of \eqref{56} we get
		\begin{equation}\label{57}
		\begin{split}
		\left(n-1-\alpha\right)^2I^2&\leq\left((n-1-\alpha)I+\alpha I_r\right)^2
		\\ 
		&\leq (1+\varepsilon)^2(p-1)^2\frac{(\alpha-2)^2}{4}I^2+\frac{C}{\varepsilon}I\norm{\nabla u}^p_{L^p(\Omega\setminus K_\delta)}.
		\end{split}
		\end{equation}
		
		Summarizing, if $\alpha\in[0,n-1)$ satisfies condition \eqref{alpha_condition}, then we can choose $ \varepsilon>0 $ in \eqref{56} or \eqref{57} such that
		\[
		\int_{K_{\delta/2}}\abs{\nabla u}^p r_y^{-\alpha}\,dx\leq C\norm{\nabla u}^p_{L^p(\Omega\setminus K_\delta)},
		\]
		for some positive constant $C$ depending only on $\Omega$, $\delta$, $p$, and $\alpha$. Finally, using \eqref{r_bound} and that $K_\delta\subset K_{\delta/2} $ we can control the integral over $ \Omega\setminus K_{\delta/2} $ with
		\[
		\int_{\Omega\setminus K_{\delta/2}}\abs{\nabla u}^p r_y^{-\alpha}\,dx\leq C\norm{\nabla u}^p_{L^p(\Omega\setminus K_\delta)},		
		\]
		proving \eqref{interior_ineq}. 
		
		Let us assume now that $ \Omega $ is a ball and $f$ is strictly positive in $(0,+\infty)$. Then Corollary 1.1 of \cite{DS} ensures that $u$ is radially symmetric and decreasing in the radius~$r$.
		Taking $y=0$, we have that $ I=I_r $. Furthermore,
		$ H^2=(n-1)\abs{A}^2 $ and $ \nabla_T\abs{\nabla u}=0 $, since $ \nabla u $ is orthogonal to the level sets. In this case, from \eqref{66} we deduce
		\begin{equation*}
		(n-1)I \leq \int_{\Omega\cap\{|\nabla u|>0\}}\abs{\nabla u}^{p}\abs{A}^2r_0^{-\alpha+2}\zeta^2\,dx,
		\end{equation*}
		instead of \eqref{99}. Therefore, under the less restrictive assumption \eqref{radial_alpha_condition}, we obtain~ \eqref{interior_ineq} with $y=0$ in the same way as in the nonradial case.
	\end{proof}

	Proposition~\ref{prop_interior_estimate} is the main tool in the proof of the interior estimate~\eqref{interior_bound}.
	In the following lemma, we introduce some boundary estimates that we will need to pass from~\eqref{interior_bound} to the global bound~\eqref{global_bound} in strictly convex domains.
	
	\vspace{3mm}
	\begin{proposition}[Castorina, Sanch\'on \cite{CasSa}]\label{prop_boundary}
		Let $ \Omega\subset\R^n$ be a bounded smooth domain, $f$ a positive $C^1$ nonlinearity, $ p\in(1,+\infty) $ and $u$ a positive regular solution to \eqref{plap_Dir}.
		
		If $ \Omega $ is strictly convex, then there exist positive constants $ \delta $ and $ \gamma $ depending only on the domain $ \Omega $, such that for every point $x$ with $ \textnormal{dist}(x,\partial\Omega)<\delta $, there exists a set $ I_x\subset\Omega$ of positive measure $ \gamma $ for which
		\begin{equation*}
		u(x)\leq u(y) \qquad \text{for every}\,\, y\in I_x.
		\end{equation*}
		In particular, 
		\begin{equation}\label{boundary_estimate}
		\norm{u}_{L^\infty(\Omega\setminus K_{\delta})}\leq \frac1\gamma\norm{u}_{L^1(\Omega)},
		\end{equation}
		where $ K_{\delta}:=\left\{x\in\Omega: \textnormal{dist}(x,\partial\Omega)\geq\delta	\right\} $.
	\end{proposition}

	The proof of this lemma --- which can be found in \cite{CasSa} --- is based on a moving planes procedure for the $p$-Laplacian developed in \cite{DS}. For this method to work, the strict convexity assumption about $\Omega$ is crucial.
	
	We can now prove our main result.
		
	\begin{proof}[Proof of Theorem \ref{thm_main}]
		Assume that there exists a nonnegative exponent $ \alpha $ satisfying~\eqref{alpha_condition} such that
		\begin{equation*}
		n-p<\alpha<n-1.
		\end{equation*}
		The existence of such an exponent $ \alpha $ depends on the values of $n$ and $ p $ and, in particular, it is ensured when we assume that $n$ and $p$ satisfy \eqref{condition} --- see Appendix~\ref{appendix} for all the details.
		
		As a consequence of Proposition \ref{prop_interior_estimate}, for every $ y\in K_\delta $ we obtain that
		\begin{equation}\label{32}
		\int_{\Omega}\abs{\nabla u(x)}^pr_y(x)^{-\alpha}\,dx\leq C\norm{\nabla u}^p_{L^p(\Omega\setminus K_\delta)},
		\end{equation}
		for some constant $ C>0 $ depending only on $ \Omega $, $ \delta $, and $p$.
		
		In the radial case, Proposition \ref{prop_interior_estimate} gives \eqref{32} with $y=0$ for some nonnegative exponent $\alpha\in(n-p,n-1) $ satisfying the less restrictive condition \eqref{radial_alpha_condition}. It can be checked that such an an exponent $ \alpha$ exists whenever $n$ and $p$ satisfy \eqref{radial_condition} --- see Remark \ref{rmk_appendix} in the appendix.
		
		Summarizing, in both the radial and the nonradial case --- under different assumptions on $ n $ and $ p $ --- we have \eqref{32} for some nonnegative $ \alpha>n-p $, and we want to deduce~\eqref{interior_bound} and~\eqref{global_bound}.
				
		In order to prove the interior bound \eqref{interior_bound}, for every point $ y\in K_\delta $ we use Lemma~7.16 of~\cite{GT} for the set $B_y:=B_{\delta/2}(y)$, obtaining
		\begin{equation*}
		\abs{u(y)-\overline{u}_{B_y}}\leq C \int_\Omega \abs{\nabla u}r_y^{1-n}\,dx.
		\end{equation*}
		Here $C$ is a positive constant depending only on $n$ and $\delta$, and $\overline{u}_{B_y}$ is the mean of $ u $ over the set $ B_y $, defined by
		\[
		\overline{u}_{B_y}:=\frac{1}{\abs{B_y}}\int_{B_y}u\,dx.
		\]
		Then, applying the Hölder inequality with exponents $ p $ and $ p' $ it follows that
		\begin{equation}\label{35}
		\abs{u(y)-\overline{u}_{B_y}}\leq C\left(\int_{\Omega}\abs{\nabla u}^pr_y^{-\alpha}\,dx\right)^\frac1p\left(\int_{\Omega}r_y^{\frac{p-np+\alpha}{p-1}}\,dx	\right)^\frac1{p'}.
		\end{equation}
		The last integral is bounded, since $ \alpha>n-p $ and
		\[
		\int_{\Omega}r_y(x)^{\frac{p-np+\alpha}{p-1}}\,dx\leq\abs{S^{n-1}}\frac{p-1}{\alpha-n+p}\,\textnormal{diam}(\Omega)^\frac{\alpha-n+p}{p-1}.
		\]
		Now, using \eqref{32} and \eqref{35} we can conclude that
		\begin{equation}\label{37}
		\norm{u}_{L^\infty(K_\delta)}\leq C\left(\norm{u}_{L^1(\Omega)}+\norm{\nabla u}_{L^p(\Omega\setminus K_\delta)}	\right),
		\end{equation}
		which is \eqref{interior_bound}, with $ C $ depending only on $ \Omega $, $p$, and $ \delta $. 

Assume now that $ \Omega $ is strictly convex, $u$ is a positive solution of problem \eqref{plap_Dir} and $f$ is positive in $(0,+\infty)$. Then, Proposition \ref{prop_boundary} gives the boundary estimate
\begin{equation}\label{90}
\norm{u}_{L^\infty(\Omega\setminus K_{2\delta})}\leq \frac{1}{2\delta}\norm{u}_{L^1(\Omega)},
\end{equation}
where $\delta$ is a positive constant that depends only on $\Omega$.
We use this bound to control $ f(u) $ in the set $ \Omega\setminus K_{2\delta}$. By interior and boundary regularity\footnote{See Theorem 1 in \cite{DiB} or Theorem 1 in \cite{T} for interior $ C^{1,\beta} $ regularity in the style of De Giorgi and Theorem 1 in \cite{L} for boundary regularity. See also Appendix E in \cite{P}.
 } 
for problem \eqref{plap_Dir}, we deduce stronger estimates in the set $ \Omega\setminus K_{\delta} $, which is contained in $\Omega\setminus K_{2\delta}$. In particular, we have $ \norm{\nabla u}_{L^\infty(\Omega\setminus K_{\delta})}\leq C$, for some constant $C$ which depends only on $ \Omega $, $ f $, $p$ and $ \norm{u}_{L^1(\Omega)} $. Combining this with \eqref{37} we obtain \eqref{global_bound}, since we also have that 
\[\norm{u}_{L^\infty(\Omega\setminus K_{\delta})}\leq\norm{u}_{L^\infty(\Omega\setminus K_{2\delta})}\leq C\norm{u}_{L^1(\Omega)}.\]

If $\Omega$ is a ball and $f$ is strictly positive in $(0,+\infty)$, then from Corollary 1.1 of \cite{DS} we know that $u$ is radially symmetric and decreasing in the radius $r$. Therefore, it is sufficient to estimate $u(0)$. From \eqref{32} with $y=0$, we obtain
\begin{equation*}
\abs{u(0)}\leq C\left(\norm{u}_{L^1(\Omega)}+\norm{\nabla u}_{L^p(\Omega\setminus K_\delta)}	\right).
\end{equation*}
Then, proceeding in the same way as in the nonradial case we deduce \eqref{global_bound}.
\end{proof}

\appendix
\section{}\label{appendix}
In this appendix we show the existence of a nonnegative $\alpha\in(n-p,n-1)$ satisfying~\eqref{alpha_condition} whenever $n$ and $p$ satisfy~\eqref{condition}, completing in this way the argument of the proof of Theorem \ref{thm_main}. We distinguish two cases.

\textbf{Case 1, $\boldsymbol{n>p}$.}
We take\footnote{The following argument gives that our condition \eqref{condition} is the optimal one for the existence of some $\alpha\in(n-p,n-1)$ satisfying \eqref{alpha_condition}. To see this in the case $n\geq4$ (otherwise it is simple), we may assume $n>p+2$, since \eqref{condition} already includes $n\leq p+2$. But then, since $\alpha\in(n-p,n-1)$ we also have $\alpha\in(2,n-1)$ and therefore, if \eqref{alpha_condition} is satisfied by some $\alpha\in(2,n-1)$, then it is also satisfied by any smaller $\alpha$ in this interval. Thus, in the argument, our choice $\alpha=n-p+\e$ for small $\e>0$ imposes non restriction.
} $\alpha=n-p+\e$ for some $\e\in(0,p-1)$, and we plug it in~\eqref{alpha_condition}, obtaining 
\begin{equation*}
4(p-1-\e)^2>(p-n+2-\e)^2(n-1)(p-1).
\end{equation*}
If the inequality holds with $\e=0$, then it also holds for an arbitrary small~$\e>0$. In this way, we are reduced to check that whenever $n$ and $p$ satisfy~\eqref{condition}, then
\begin{equation}\label{alpha_ineq_n>p}
4(p-1)>(p-n+2)^2(n-1).
\end{equation}
This inequality is cubic in $n$, but quadratic in $p$. Solving it with respect to $p$, and exploiting also a surprising cancellation in the discriminant, we obtain
\[
\frac{n^2-5n+8}{n-1}<p<n.
\]
Observe that this forces $n>2$. For $n\geq3$, we solve it with respect to $n$ and find
\begin{equation}\label{n_bound}
\frac12\left(p+5-\sqrt{(p-1)(p+7)}\right)<n<\frac12\left(p+5+\sqrt{(p-1)(p+7)}\right).
\end{equation}
If $n=3$, both inequalities hold true for every $p<3=n$.
If $n\geq4$ instead, the lower bound in \eqref{n_bound} is always verified, and the upper bound on~$n$ is the one appearing in~\eqref{condition}. 




\textbf{Case 2, $\boldsymbol{n\leq p}$.} In this case, inequality \eqref{alpha_condition} reads
\begin{equation}\label{alpha_n<p_ineq}
4(n-1-\alpha)^2>(\alpha-2)^2(p-1)^2.
\end{equation}
For $n=2,3$ one can directly check that no nonnegative solutions $\alpha\in(n-p,n-1)$ exist. Indeed, when $n=2$ we can take the square root of \eqref{alpha_n<p_ineq} and check that the solutions $\alpha$ are either strictly negative or greater than 1. When $n=3$ instead, \eqref{alpha_n<p_ineq} contradicts the assumption $p\geq n=3$.

For $n\geq4$, we are going to see that, for every $p\geq n$, there exists a nonnegative $\alpha\in(n-p,n-1)$ satisfying \eqref{alpha_n<p_ineq}. For this, it suffices to look for $\alpha$ belonging to~$(2,n-1)$. We can now take the square root of \eqref{alpha_n<p_ineq} and solve the inequality with respect to $\alpha$. In this way, we find
\[
\alpha<\frac{2(n+p-2)}{p+1}
\]
and one can easily check that $2(n+p-2)/(p+1)>2$ for all $p>1$, since we are assuming $n\geq4$.

\begin{remark}\label{rmk_appendix}
	The same ideas --- including the same choice of $\alpha$ when $n>p$ --- can be used to check that in the radial case there exists a nonnegative $\alpha\in(n-p,n-1)$ satisfying~\eqref{radial_alpha_condition} whenever $n$ and $p$ satisfy \eqref{radial_condition}. The only difference is that in the case $n>p$ we get an inequality which is quadratic in $n$ and cubic in~$p$, and we can directly solve it with respect to $n$, finding $p<n<p+4p/(p-1)$.
\end{remark}

\vspace{0.1cm}
\section*{Acknowledgments} The author thanks Xavier Cabré for his guidance and useful discussions on the topic of this paper.


\vspace{0.1cm}
\begin{thebibliography}{99}
\bibitem{B}
Brezis, H. 
Is there failure of the Inverse Function Theorem?,
in {\it ``Morse Theory, Minimax Theory and Their Applications to Nonlinear Differential Equations''}, {\bf 23-33}, New Stud. Adv.	Math., 1, Int. Press, Somerville, MA (2003).
	
\bibitem{BetA}
Brezis, H.; Cazenave, T.; Martel Y.; Ramiandrisoa A.
Blow up for $ u_t-\Delta u =g(u) $ revisited, 
{\it Adv. Differential Equations, {\bf 1}} (1996), 73-90.

\bibitem{BV}
Brezis, H.; Vazquez, J. L.
Blow-up solutions of some nonlinear elliptic problems,
{\it Rev. Mat. Complut. (2)}, 10 (1997), 443-469.

\bibitem{C}
Cabr\'{e}, X.
Regularity of minimizers of semilinear elliptic problems up to dimension 4,
{\it Comm. Pure Appl. Math. (10)}, 63 (2010), 1362-1380.

\bibitem{C2}
Cabr\'{e}, X.
Boundedness of stable solutions to semilinear elliptic equations: a survey,
{\it Adv. Nonlinear Stud. (2)}, 17 (2017), 355-368.

\bibitem{C1}
Cabr\'{e}, X.
A new proof of the boundedness results for stable solutions to semilinear elliptic equations,
{\it preprint}, arXiv:1907.05253v2 (2019).

\bibitem{CCa}
Cabr\'{e}, X.; Capella, A.
Regularity of radial minimizers and extremal solutions of semilinear elliptic equations,
{\it J. Funct. Anal.}, 238 (2006), 709–733.

\bibitem{CCaSa}
Cabr\'{e}, X.; Capella, A.; Sanch\'on, M.
Regularity of radial minimizers of reaction equations involving the $p$-Laplacian,
{\it Calc. Var. Partial Differential Equations}, 34 (2009), 475-494.

\bibitem{CFRS}
Cabr\'e, X.; Figalli, A.; Ros-Oton, X.; Serra, J.
Stable solutions to semilinear elliptic equations are smooth up to dimension 9,
{\it preprint}, arXiv:1907.09403 (2019).

\bibitem{CM}
Cabr\'{e}, X.; Miraglio, P.
Universal Hardy-Sobolev inequalities on hypersurfaces of Euclidean space,
{\it forthcoming}.

\bibitem{CMS}
Cabr\'{e}, X.; Miraglio, P.; Sanchón, M.
Optimal regularity of stable solutions to nonlinear equations involving the $p$-Laplacian,
{\it forthcoming}.

\bibitem{CSa}
Cabr\'{e}, X.; Sanch\'on, M.
Semi-stable and extremal solutions of reaction equations involving the $p$-Laplacian,
{\it Comm. Pure Appl. Math. (1)}, 6 (2007), 43-67.

\bibitem{CasSa}
Castorina, D.; Sanch\'on, M.
Regularity of stable solutions of $p$-Laplace equations through geometric Sobolev type inequalities,
{\it J. Eur. Math. Soc.}, 17 (2015), 2949-2975.

\bibitem{CR}
Crandall, M. G.; Rabinowitz, P. H.
Some continuation and variational methods for positive solutions of nonlinear elliptic eigenvalue problems,
{\it Arch. Rational Mech. Anal.}, 58 (1975), 207–218.

\bibitem{DS}
Damascelli, L.; Sciunzi, B.
Regularity, monotonicity and symmetry of positive solutions of $m$-Laplace equations,
{\it J. Differential Equations (2)}, 206 (2004), 483-515.

\bibitem{DiB}
DiBenedetto, E.
$C^{1+\alpha }$ local regularity of weak solutions of degenerate elliptic equations,
{\it Nonlinear Anal. (7)}, 8 (1983), 827-850.

\bibitem{D}
Dupaigne, L.
Stable solutions of elliptic partial differential equations,
{\it Chapman and Hall/CRC, 143}, CRC Press (2011).

\bibitem{FSV1}
Farina, A.; Sciunzi, B.; Valdinoci, E.
Bernstein and De Giorgi type problems: new results via a geometric approach,
{\it Ann. Sc. Norm. Super. Pisa Cl. Sci. (5)}, 7 (2008), 741-791.

\bibitem{FSV}
Farina, A.; Sciunzi, B.; Valdinoci, E.
On a Poincar\'{e} type formula for solutions of singular and degenerate elliptic equations,
{\it Manuscripta Math. (3-4)}, 132 (2010), 335-342.

\bibitem{GP}
Garc\'{i}a Azorero, J.; Peral, I.
On an Emden-Fowler type equation,
{\it Nonlinear Anal. (11)}, 18 (1992), 1085-1097.

\bibitem{GPP}
Garc\'{i}a Azorero, J.; Peral, I.; Puel, J.-P.
Quasilinear problems with exponential growth in the reaction term,
{\it Nonlinear Anal. (4)}, 22 (1994), 481-498.

\bibitem{GT}
Gilbarg, D.; Trudinger, N. S.
Elliptic partial differential equations of second order,
{\it Reprint of the 1998 edition. Classics in Mathematics.} Springer-Verlag, Berlin (2001).

\bibitem{JL}
Joseph, D. D.; Lundgren, T. S.
Quasilinear Dirichlet problems driven by positive sources,
{\it Arch. Rational Mech. Anal.}, 49 (1973), 241–269.

\bibitem{L}
Lieberman, G. M.
Boundary regularity for solutions of degenerate elliptic equations,
{\it Nonlinear Anal. (11)}, 12 (1988), 1203-1219.

\bibitem{N}
Nedev, G.
Regularity of the extremal solution of semilinear elliptic equations,
{\it C. R. Acad. Sci. Paris	Sér. I Math.}, 330 (2000), 997–1002.

\bibitem{P}
Peral, I.
Multiplicity of solutions for the $p$-Laplacian,
{\it Lecture notes International Center for Theoretical Physics, Trieste (1997)}, http://matematicas.uam.es/~ireneo.peral/ICTP.pdf.

\bibitem{San}
Sanchón, M. 
Boundedness of the extremal solution of some $p$-Laplacian problems,
{\it Nonlinear Anal.}, 67 (2007), 281-294.

\bibitem{San1}
Sanchón, M.
Regularity of the extremal solution of some nonlinear elliptic problems involving the $p$-Laplacian,
{\it Potential Anal. (3)}, 27 (2007), 217-224.

\bibitem{SZ1}
Sternberg, P.; Zumbrun, K.
Connectivity of phase boundaries in strictly convex domains,
{\it Arch. Rational Mech. Anal. (4)}, 141 (1998), 375-400.

\bibitem{SZ2}
Sternberg, P.; Zumbrun, K.
A Poincar\'{e} inequality with applications to volume-constrained area-minimizing surfaces,
{\it J. Reine Angew. Math.}, 503 (1998), 63-85.

\bibitem{T}
Tolksdorf, P.
Regularity for a more general class of quasilinear elliptic equations,
{\it J. Differential Equations (1)}, 51 (1984), 126-150.	

\bibitem{V}
Villegas, S.
Boundedness of extremal solutions in dimension 4,
{\it Adv. Math.}, 235 (2013), 126–133.
\end{thebibliography}
\end{document}